\documentclass{amsart}

\usepackage[sort&compress, round]{natbib}
\usepackage{amsmath, amssymb, amsthm, amscd, verbatim}
\usepackage{enumitem}
\usepackage[usenames,dvipsnames]{color}
\usepackage{tikz-cd}

\usepackage{hyperref}
\hypersetup{
           breaklinks=true,   
           colorlinks=true,   
			citecolor = Black
        }

\newcommand{\R}  {{\mathbb R}}
\newcommand{\C}  {{\mathbb C}}
\newcommand{\N}  {{\mathbb N}}

\newcommand{\bs} {{\boldsymbol{\sigma}}}
\newcommand{\bt} {{\boldsymbol{\tau}}}
\newcommand{\bo} {{\boldsymbol{\omega}}}
\newcommand{\bb} {{\boldsymbol{\beta}}}

\newcommand{\bx} {{\boldsymbol{x}}}
\newcommand{\by} {{\boldsymbol{y}}}
\newcommand{\bz} {{\boldsymbol{z}}}
\newcommand{\bee}{{\boldsymbol{e}}}
\newcommand{\be} {{\boldsymbol{1}}}
\newcommand{\bc} {{\boldsymbol{c}}}
\newcommand{\bU} {{\boldsymbol{U}}}
\newcommand{\bu} {{\boldsymbol{u}}}
\newcommand{\bv} {{\boldsymbol{v}}}
\newcommand{\bn} {{\boldsymbol{0}}}
\newcommand{\cB} {{\mathcal{B}}}

\newcommand{\X}  {{\mathfrak X}}

\newcommand{\nnorm}{n^{\mathrm{norm}}}
\newcommand{\nabs} {n^{\mathrm{abs}}}

\newcommand{\eps}  {\varepsilon}
\renewcommand{\phi}{\varphi}

\DeclareMathOperator{\cost}{cost}
\DeclareMathOperator{\decay}{decay}
\DeclareMathOperator{\dec}{dec}
\DeclareMathOperator{\Act}{Act}

\numberwithin{equation}{section}

\theoremstyle{plain}
\newtheorem{lemma}{Lemma}[section]
\newtheorem{theo}[lemma]{Theorem}
\newtheorem{cor}[lemma]{Corollary}
\theoremstyle{definition}
\newtheorem{rem}[lemma]{Remark}

\begin{document}

\title[]{Multi- and Infinite-variate Integration and $L^2$-Approximation
on Hilbert Spaces with Gaussian
Kernels}

\author[Gnewuch]
{M.~Gnewuch}
\address{
Institut f\"ur Mathematik\\
Universit\"at Osnabr\"uck\\
Albrechtstra{\ss}e \ 28A\\
49076 Osnabr\"uck\\ 
Germany}
\email{michael.gnewuch@uni-osnabrueck.de}

\author[Ritter]
{K.~Ritter}
\address{Fachbereich Mathematik\\
RPTU Kaisers\-lautern\\
Postfach 3049\\
67653 Kaiserslautern\\
Germany}
\email{klaus.ritter@math.rptu.de}

\author[R\"u\ss mann]
{R.~R\"u\ss mann}
\address{Fachbereich Mathematik\\
RPTU Kaisers\-lautern\\
Postfach 3049\\
67653 Kaiserslautern\\
Germany}
\email{robin.ruessmann@math.rptu.de}

\date{December 05, 2025}

\begin{abstract}
We study integration and $L^2$-approximation in the worst-case
setting for deterministic linear algorithms based on 
function evaluations.
The underlying function space is
a reproducing kernel Hilbert space with a Gaussian kernel
of tensor product form. 
In the infinite-variate case, for both computational problems,
we establish matching upper and lower bounds
for the polynomial convergence rate of the $n$-th minimal error.
In the multivariate case, we improve several tractability results for the integration problem.
For the proofs, we establish the following transference result together with an explicit construction: 
Each of the computational problems on a space with a Gaussian kernel
is equivalent on the level of algorithms to the same 
problem on a Hermite space with suitable parameters.

\end{abstract}
\maketitle

\section{Introduction}
We consider integration and $L^2$-approximation of functions $f \colon
\X \to \R$ of $d$ variables in the cases $d \in \N$ and $d = \infty$,
where the underlying measure $\mu$ is the $d$-fold product of
the univariate standard normal distribution $\mu_0$.
The functions $f$ are assumed to be
elements of a reproducing kernel Hilbert space $H(M)$ with a reproducing 
kernel $M$ of 
tensor product form with domain $\X:=\X(M)$. For $d \in \N$ the 
domain is given by $\X := \R^d$ (and $\mu$ is the
$d$-dimensional standard normal distribution), while 
for $d=\infty$ the domain $\X$ is a proper
subset of the sequence space $\R^\N$ with $\mu(\X)=1$.  
This setting, with $\mu_0$ as before or with
$\mu_0$ being the uniform distribution on $[0,1]$,
is frequently studied in the literature on multivariate
or infinite-variate integration and approximation. 

Throughout this paper, we use the following notation.
We put $J := \{1,\dots,d\}$ if $d \in \N$ and
$J := \N$ if $d=\infty$, and 
$h_\nu$ denotes the univariate Hermite polynomial of 
degree $\nu \in \N_0$, normalized in $L^2(\mu_0)$.

We study two particular kinds of reproducing kernels.
Primarily, we are interested in Gaussian kernels
\[
L_\bs(\bx,\by) := \prod_{j \in J} 
\exp \left( -\sigma_j^2 \cdot (x_j-y_j)^2\right)
\]
with $\bx,\by \in \X(L_\bs)$,
which are specified by sequences $\bs := (\sigma_j)_{j \in J}$ of shape 
parameters $\sigma_j > 0$. 
Additionally, we consider Hermite kernels
\[
K_\bb(\bx,\by) := \prod_{j \in J}
\Biggl(\,
\sum_{\nu=0}^\infty \beta_j^\nu \cdot h_\nu(x_j) \cdot h_\nu(y_j)
\Biggr)
\]
with $\bx,\by \in \X(K_\bb)$, which 
are specified by sequences $\bb := (\beta_j)_{j \in J}$ of base 
parameters $0 < \beta_j < 1$.
See Section~\ref{s2} for the study of
the function spaces $H(L_\bs)$ and $H(K_\bb)$ and, in the case $d=\infty$,
of the corresponding domains $\X(L_\bs)$ and $\X(K_\bb)$.
Here we only mention the following monotonicity properties:
If $\bs \leq \bs^\prime$, then
$\X(L_{\bs^\prime}) \subseteq \X(L_{\bs})$
and the restriction $f \mapsto f|_{\X(L_{\bs^\prime})}$
defines a continuous mapping from
$H(L_\bs)$ to $H(L_{\bs^\prime})$;
analogously, if $\bb \leq \bb^\prime$, then
$\X(K_{\bb^\prime}) \subseteq \X(K_{\bb})$
and the restriction $f \mapsto f|_{\X(K_{\bb^\prime})}$
defines a continuous mapping from
$H(K_\bb)$ to $H(K_{\bb^\prime})$.

For both kernels, $M := L_\bs$ and $M := K_\bb$, and for both problems, 
integration and $L^2$-approximation, we analyze deterministic linear 
algorithms that are based on a fixed number of function
evaluations in the worst-case setting on the unit ball in 
$H(M)$. Concerning the definition of cost we distinguish
two cases: For finite $d$ the cost per function evaluation at any
point $\bx \in \X$ is
one, the most natural choice, so that the worst-case cost of an algorithm is equal
to the number of function evaluations.
For $d = \infty$ it is widely accepted that
the analogous definition of cost is too generous and
non-predictive for computational practice.
Therefore we employ the unrestricted subspace 
sampling model, which has been introduced in
\citet{KuoEtAl10} to define the worst-case cost of algorithms 
for computational problems on spaces of functions of infinitely many 
variables.  Accordingly, function evaluation is only permitted at points 
$\bx := (x_j)_{j \in J} \in \X$ with finitely many non-zero components
$x_j$, and $\cost(\bx)$, the cost for the evaluation at $\bx$, depends
non-decreasingly on the number of non-zero components.
The worst-case setting is presented in detail in Section~\ref{s2a}.

A key quantity in the analysis is
the $n$-th minimal error $e_n(M)$
for integration and for $L^2$-approximation, which is, roughly speaking,
the smallest worst-case error on the unit ball in $H(M)$
that can be achieved by any algorithm with worst-case cost at most 
$n \in \N$.

In this framework integration and $L^2$-approximation on $H(L_\bs)$ 
with $d \in \N$
are studied in
\citet{FHW2012},
\citet{KW2012},
\citet{KSW2017},
\citet{KS19},
and
\citet{KOG21},
see also
\citet{SW2018}.
On $H(K_\bb)$ with $d \in \N$ these problems are studied in
\citet{IKLP15}
and in
\citet{IKPW16b},
see also \citet{IKPW16a}.
Actually, the kernels $K_\bb$, which define spaces $H(K_\bb)$ of
functions of infinite smoothness, are only particular elements of 
the broader class of Hermite kernels, which may also define
spaces of functions of finite smoothness. 
Here we refer to
\citet{IL15}, 
\citet{DILP18},
and \citet{LPE22}
for results on integration and $L^2$-approximation with $d \in \N$.

For $d = \infty$ we are not aware of results on
integration and $L^2$-approximation on spaces
with Gaussian kernels, while spaces with Hermite kernels, including the 
kernels $K_\bb$, have been studied for both problems in
\citet{GneEtAl24}.

Numerous models in probability theory 
are ultimately based on a sequence of independent random variables, 
and in every such model
expectations may be represented as integrals with respect to product 
measures.
This motivates the study of integration of functions $f$ defined 
on the sequence space $\R^\N$ or subspaces thereof.
Let us mention, for instance, the 
L\'evy-Ciesielski (or Brownian bridge) representation of Brownian motion,
which is based on an independent sequence of standard normally
distributed random variables and thus leads to the measure
$\mu$ that is studied in this paper. In this setting, the results
derived in this paper may potentially be used to compute
expectations of smooth functionals of Brownian motion.

In this paper we first of all
establish a transference result for integration and a transference 
result for $L^2$-approximation, which 
yield
for each of these problems 
and suitably related spaces $H(L_\bs)$ and $H(K_\bb)$
the equivalence
on the level of algorithms.
The transference results only require that
\begin{equation}\label{g54}
\sum_{j \in J} \sigma_j^2 < \infty,
\end{equation}
which is trivially satisfied in the case $d \in \N$.
In the proofs we employ a particular
isometric isomorphism between Hilbert spaces with Gaussian and with
Hermite kernels, see Theorem~\ref{t1}, which is due to \citet{GHHR2021}.

For integration with any $d$ the transference result reads as follows. 
Assume that the shape and base parameters are related by
\[
1-\beta_j = \frac{1}{1+2\sigma_j^2}
\]
for every $j \in J$. Moreover, 
let $e_0(M)$ denote the norm of the integration functional on $H(M)$.
Then for every $n \in \N$ there exists a bijection
$A \mapsto B$ on the set of all $n$-point quadrature formulas
such that 
\[
e(A,L_{\bs}) = e_0(L_\bs) \cdot e(B, K_\bb)
\]
for the worst-case errors $e(A,L_\bs)$ and $e(B,K_\bb)$
of $A$ on the unit ball in $H(L_\bs)$ and of
$B$ on the unit ball in $H(K_\bb)$. See Theorem~\ref{t3}.
The bijection is given
explicitly in Remark~\ref{r7}, and the worst-case cost of $A$
and $B$ coincide also in the case $d=\infty$. Consequently,
the minimal errors satisfy
\[
e_n(L_{\bs}) = e_0(L_\bs) \cdot e_n(K_\bb)
\]
for every $n \in \N$.

For $L^2$-approximation with any $d$ an analogous transference result is 
established in Theorem~\ref{t2} and Remark~\ref{r5}, assuming
that the shape and base parameters are related by
\[
1 - \beta_j = \frac{2}{1+ \left(1+8 \sigma_j^2 \right)^{1/2}}
\]
for every $j \in J$. 

Let $d = \infty$.
We combine the transference results 
with known results for the decay of the minimal errors
$e_n(K_\bb)$ from \citet{GneEtAl24} to determine the decay of
the minimal errors $e_n(L_\bs)$.
For any non-increasing sequence $\bee:=(e_n)_{n \in \N}$ 
of positive real numbers its decay is given by
\begin{equation}\label{g4}
\decay(\bee) := 
\sup\{\tau > 0 \colon 
\sup_{n \in \N} \left(e_n \cdot n^\tau \right) < \infty\}
\end{equation}
with the convention that $\sup \emptyset := 0$. Put
\[
\rho:=   \liminf_{j\to \infty}
\frac{\ln(1/\sigma_j^2)}{\ln(j)},
\]
and assume that \eqref{g54} is satisfied, which implies $1 \leq \rho 
\leq \infty$. Under a very mild assumption on the cost function
$\bx \mapsto \cost(\bx)$, see \eqref{g82}, we obtain
\[
\decay \left( (e_n(L_\bs))_{n \in \N} \right)
= \frac{1}{2} (\rho - 1),
\]
see Theorem~\ref{Thm:InfDimInt} for integration and 
Theorem~\ref{Thm:InfDimApp} for $L^2$-approximation. 
Algorithms showing that the polynomial decay rate of
the minimal errors $e_n(L_\bs)$ is at least $(\rho-1)/2$
are obtained by 
applying the transference results to properly chosen
multivariate decomposition methods on the spaces $H(K_\bb)$,
see Section~\ref{SUBSEC:integr_IMV}.

Now, we turn to the case $d \in \N$. In the sequel
we survey known results from the literature 
together with the findings from this paper for 
the integration problem. To this end we fix 
sequences $(\sigma_j)_{j \in \N}$ and $(\beta)_{j \in \N}$ of shape
parameters and base parameters, respectively, and we either have
$M_d := L_{(\sigma_1, \dots, \sigma_d)}$ for all $d \in \N$
or $M_d := K_{(\beta_1, \dots, \beta_d)}$ for
all $d \in \N$. 

At first, we 
note that exponential convergence of $e_n(M_d)$ 
holds for every fixed dimension $d$:
For every $d \in \N$ there exists a constant $c>0$ such 
that
\begin{equation}\label{g50}
e_n(M_d) \leq \exp \bigl( -c \cdot n^{1/d}\bigr)
\end{equation}
for all $n \geq 1$.
For $d = 1$ we have a matching lower bound, i.e.,
there exists a constant $c^\prime>0$ 
such that
\begin{equation}\label{g51}
e_n(M_1) \geq \exp \bigl( -c^\prime \cdot n\bigr)
\end{equation}
for all $n \geq 1$. For $d \geq 2$ it is only known that
\begin{equation}\label{g52}
\limsup_{n \to \infty} \left(e_n(M_d) \cdot
\exp \bigl( c \cdot n^{p}\bigr) \right) = \infty
\end{equation}
for all $p > 1/d$ and $c > 0$.

For Hermite kernels all of these results
have been established in \citet{IKLP15}. 
For Gaussian kernels the situation is as follows:
The upper bound \eqref{g50} is due to
\citet{KSW2017} and \citet{KOG21} under the additional assumptions
$\max_{1 \leq j \leq d}\sigma_j < 1/2$ and
$\sigma_1 = \dots = \sigma_d$, respectively;
from the latter one can easily deduce \eqref{g50} in full generality.
The lower bound \eqref{g51} follows from Theorem~\ref{t4} and improves the lower bound 
from \citet{KSW2017}, which is super-exponentially small in $n$.
Moreover, \eqref{g52} is due to
\citet{KSW2017} in the case $\max_{1 \leq j \leq d}\sigma_j < 1/2$, 
and the monotonicity property mentioned above immediately yields 
\eqref{g52} in full generality.

Next, we discuss tractability results, which deal with
the information complexity $\nnorm(\eps,M_d)$ with $0 < \eps < 1$
and $d \in \N$ for the normalized error criterion.
Roughly speaking, $\eps \mapsto \nnorm(\eps,M_d)$ is the inverse
function to $n \mapsto e_n(M_d)/e_0(M_d)$.
Concerning tractability, the exponential convergence
\eqref{g50} suggests to control the behavior of 
$\ln(\nnorm(\eps,M_d))$ as $d+\eps^{-1}\to\infty$ in terms
of $\ln(\eps^{-1})$ and a power of $d$, which leads to
the following definition, see \citet{KSW2017}.
Exponential convergence $(t,\kappa)$-weak tractability, for short
EC-$(t,\kappa)$-WT, holds for $(M_d)_{d \in \N}$ if
\[
\lim_{d+\eps^{-1}\to\infty}
\frac{\ln(\nnorm(\eps,M_d))}{d^t+\ln(\eps^{-1})^{\kappa}} = 0.
\]

At first, we address the case $(t,\kappa)=(1,1)$, where we put
$\gamma_j := \sigma_j$ in the Gaussian case and $\gamma_j := \beta_j$ 
in the Hermite case and assume that 
$(\gamma_j)_{j \in \N}$ 
is non-increasing to obtain the following result: 
EC-$(1,1)$-WT is equivalent to $\lim_{j \to \infty} \gamma_j = 0$.
For Hermite kernels $\lim_{j \to \infty} \beta_j = 0$ 
has been established as a necessary condition in \citet{IKLP15} and 
as a sufficient condition in \citet{IKPW16b}. 
For Gaussian kernels a characterization of
EC-$(1,1)$-WT has not been available. Under the additional
assumption that $\bs$ is bounded by a constant less than $1/2$,
it was known that $\lim_{j \to \infty} \sigma_j = 0$ is a necessary
condition, while exponential convergence of the shape parameters
towards zero had been established as sufficient condition, see
\citet{KSW2017}.

To study the case $t>1$ and $\kappa>0$ we put 
$\gamma_j := \sigma_j$ in the Gaussian case and 
$\gamma_j := 1/(1-\beta_j)$ in the Hermite case 
to obtain the following result: 
If $ (\gamma_j)_{j \in \N}$ is polynomially bounded, 
then EC-$(t,\kappa)$-WT holds for all $t > 1$ and $\kappa>0$.
See Theorems~\ref{t6} and \ref{t8} for more
general results that permit exponentially growing shape parameters
$\sigma_j$ and exponentially shrinking gaps $1-\beta_j$ for the
base parameters $\beta_j$. 
For Hermite kernels the case $t>1$ has not been studied before, and
for Gaussian case it was only known that
EC-$(t,1)$-WT holds for all $t > 1$ and $\kappa\geq 1$, if $\bs$ is non-increasing
and bounded by a constant less than $1/2$, see \citet{KSW2017}.

We close the introduction with some remarks concerning our notation.
For any sequence $\bo := (\omega_j)_{j \in J}$ with
$\omega_j \geq 0$ for every $j \in J$
we use $\ell^2(\bo)$ to denote the corresponding weighted
$\ell^2$-space, i.e.,
\begin{equation}\label{g45}
\ell^2(\bo) := 
\biggl\{ (x_j)_{j \in J} \in \R^J \colon 
\sum_{j \in J} \omega_j \cdot x_j^2 < \infty \biggr\}.
\end{equation}

Let $\bt:=(\tau_j)_{j\in J}$ and
$\bx:=(x_j)_{j\in J}$ be sequences in $\R$ and $z \in \R$. 
We put $\bt \bx := (\tau_j x_j)_{j \in J}$ and
$\bt^{z} := (\tau_j^{z})_{j\in J}$ if $\bt$ is positive, 
i.e., $\tau_j > 0$ for every $j \in J$.
Moreover, $\be-\bt := (1-\tau_j)_{j \in \N}$.
If $\bt$ is positive and $\sum_{j \in J} |\tau_j-1| < \infty$ then
\begin{equation}\label{g44}
\bt_\ast := \prod_{j \in J} \tau_j > 0.
\end{equation}

For any Hilbert space $H(M)$ with reproducing kernel $M$ we use 
$\|\cdot\|_{H(M)}$ and $\langle \cdot,\cdot\rangle_{H(M)}$ to denote 
the corresponding norm and scalar product, respectively.

\section{The Function Spaces}\label{s2}

In this paper we study two different types of
reproducing kernel Hilbert spaces (RKHSs), namely spaces with 
Gaussian kernels and spaces
with a particular kind of Hermite kernels.  
For Gaussian and for
Hermite kernels tensor products are used
to proceed from the univariate case $d=1$ to the multivariate case
$d \in \N$ and the infinite-variate case $d = \infty$.
In this section
we mainly present known results from \citet{SHS2006}, see also 
\citet{SC08} and \citet{Minh10}, and from
\citet{GHHR2021}, most importantly a
particular one-to-one correspondence between both types of RKHSs. 
For basic facts about RKHSs and their kernels we refer to \citet{Aro50}
and \citet{PR16}.

\subsection{Univariate Gaussian Kernels}\label{SUBSEC:univ_gauss}

The univariate Gaussian kernel $\ell_\sigma$ with shape parameter 
$\sigma > 0$ is defined by
\[
\ell_\sigma (x,y) := \exp \left( - \sigma^2 \cdot (x-y)^2 \right)
\]
for $x,y \in \R$. Despite a different use in stochastic analysis,
the Hilbert space $H(\ell_\sigma)$ will be called a Gaussian space
throughout this paper.

For the analysis of 
Gaussian spaces we refer to \citet{SHS2006}, see also 
\citet[Sec.~4.4]{SC08}. In particular, each function
$f\in H(\ell_\sigma)$ is the real part of an entire function $g$ 
restricted to the real line, where $g$ belongs to the complex 
reproducing kernel Hilbert space with kernel $\ell_\sigma$ 
extended to $\C$ in the obvious way. 
Since 
\begin{align*}
|f(x) - f(y)|^2 &=
\langle f, 
\ell_\sigma (\cdot,x) - \ell_\sigma(\cdot,y)\rangle_{H(\ell_\sigma)}^2\\
&\leq \|f\|^2_{H(\ell_\sigma)} \cdot 
\|\ell_\sigma (\cdot,x) - \ell_\sigma(\cdot,y)\|^2_{H(\ell_\sigma)}
\\
&= 2 \, \|f\|^2_{H(\ell_\sigma)} \cdot (1 - \ell_\sigma(x,y))
\end{align*}
for all $x,y \in \R$, the functions from
$H(\ell_\sigma)$ are bounded.
No Gaussian space contains any non-zero polynomial, and
obviously $\ell_\sigma$ is translation-invariant, i.e.,
$\ell_\sigma(x,y) = \ell_\sigma(x-y,0)$.
Moreover, we have the following monotonicity property
of Gaussian spaces:
\begin{equation}\label{g12}
\sigma_1 < \sigma_2  \Rightarrow H(\ell_{\sigma_1})
\subsetneq H(\ell_{\sigma_2})
\end{equation}
with a non-compact continuous identical embedding of norm 
$\sqrt{\sigma_2/\sigma_1}$.

\subsection{Univariate Hermite Kernels}\label{sec:uhk}

Let $0 < \alpha_0 \leq \alpha_1 \leq \dots$ with
\[
\sum_{\nu \in \N} \alpha_\nu^{-1} \cdot \nu^{-1/2} < \infty.
\]
Then 
\[
k(x,y) := \sum_{\nu \in \N_0} \alpha_\nu^{-1} 
\cdot h_\nu(x) \cdot h_\nu(y),
\]
which is well defined for all $x,y \in \R$, yields a
reproducing kernel. The mapping $k$ and the corresponding Hilbert
space $H(k)$ are called a Hermite kernel and a Hermite space,
respectively. 

For the analysis of 
Hermite spaces we refer to \citet{IL15}, \citet{GHHR2021},
and \citet{LPE22}. In particular,
\begin{equation}\label{g10}
H(k) = \left\{ \sum_{\nu \in \N_0} c_\nu \cdot h_\nu \colon
\text{$c_\nu \in \R$ with $\sum_{\nu \in \N_0} \alpha_\nu \cdot
c_\nu^2 < \infty $}
\right\}
\end{equation}
and
\begin{equation}\label{g11}
\|f\|_{H(k)}^2 = \sum_{\nu \in \N_0} \alpha_\nu \cdot 
\left(\int_{\R} f \cdot h_\nu \, d \mu_0\right)^2
\end{equation}
for $f \in H(k)$. Consequently, every Hermite space contains
all polynomials. Hermite kernels are not translation invariant,
since $\lim_{x \to \infty} k(x,x) = \infty$.

In the present paper we consider Hermite spaces with 
\[
\alpha_\nu := \beta^{-\nu}
\]
for $\nu \in \N_0$ with a parameter $0 < \beta < 1$,
which will be called the base parameter of the kernel $k_\beta$
given by
\[
k_\beta (x,y) := 
\sum_{\nu \in \N_0} \beta^\nu \cdot h_\nu(x) \cdot h_\nu(y)
\]
for $x,y \in \R$. The elements of $H(k_\beta)$ are real analytic
functions, and we have the following monotonicity property
of Hermite spaces, which follows immediately from \eqref{g10} and
\eqref{g11}:
\begin{equation}\label{eq:embedding_hermite}
\beta_1 < \beta_2  \Rightarrow H(k_{\beta_1}) \subsetneq H(k_{\beta_2})
\end{equation}
with a compact continuous identical embedding of norm one.

\subsection{Tensor Products of Gaussian Kernels}

In this paper we study $d$-fold tensor products of Gaussian kernels and of
Hermite kernels. For the Gaussian kernels we consider a sequence 
$\bs := (\sigma_j)_{j \in J}$ of shape parameters, i.e., $\sigma_j
> 0$.

In addition to the weighted $\ell^2$-space $\ell^2(\bo)$, see
\eqref{g45}, we also consider the space $\ell^\infty$ of all
bounded sequences in $\R$.
The following observation will be used
in the study of the domains of tensor products of Gaussian and of
Hermite kernels in the case $d=\infty$.

\begin{lemma}\label{l4}
For every positive sequence $\bo := (\omega_j)_{j\in\N}$ 
with $\sum_{j \in \N} \omega_j < \infty$ we have
\[
\ell^\infty \subsetneq \ell^2(\bo) \subsetneq \R^\N
\qquad \text{and} \qquad
\mu(\ell^2(\bo)) = 1.
\]
\end{lemma}

\begin{proof}
The first statement obviously holds true, and the second statement
follows from
\[
\int_{\R^\N} \sum_{j \in \N} \omega_j \cdot x_j^2 \, d \mu(\bx) 
= \sum_{j \in \N} \omega_j \cdot \int_{\R} x_j^2 \, d \mu_0(x_j)
=
\sum_{j \in \N} \omega_j < \infty,
\]
which implies that $\sum_{j \in \N} \omega_j \cdot x_j^2 <
\infty$ holds for $\mu$-almost every $\bx \in \R^\N$.
\end{proof}

For every sequence $\bs$ 
and $\bx, \by \in \R^J$ we define 
\begin{equation}\label{eq:kernel_full_domain}
\widetilde{L}_\bs (\bx,\by) 
:= \prod_{j \in J} \ell_{\sigma_j} (x_j,y_j)
= \exp(- \sum_{j \in J} \sigma_j^2 \cdot (x_j-y_j)^2)
\end{equation}
with the convention that $\exp(-\infty) := 0$.

\begin{lemma}\label{l5}
For every sequence $\bs$, we have $H(\widetilde{L}_\bs) \subseteq L^2(\mu)$ with a compact identical 
embedding.
\end{lemma}

\
\begin{proof}
Since $\widetilde{L}_\bs(\bx,\bx) = 1$ for every $\bx \in \R^J$, 
the kernel $\widetilde{L}_\bs$ has finite trace, i.e.,
\[
\int_{\R^J} \widetilde{L}_\bs (\bx,\bx) \, d \mu (\bx) 
< \infty.
\]
The latter implies the statement of the lemma, see, e.g., \citet[Lemma~2.3]{StSc2012}.
\end{proof}

The summability property
\begin{equation}\label{g1a}
\sum_{j \in J} \sigma_j^2 < \infty
\end{equation}
will play an important role in the sequel.
For $d<\infty$ this property
is trivially satisfied and $\R^d=\ell^2(\bs^2)$.
In subsequent sections, we consider integration and $L^2$-approximation, which are trivial if \eqref{g1a} does not hold, see Lemma~\ref{l5i} and Remark~\ref{rem:norm_zero}.

In any case, due to Lemma~\ref{l5}, the integration functional  $f \mapsto
\int_{\R^J} f \, d \mu$
is well-defined and continuous on 
$H(\widetilde{L}_\bs)$, and its representer $h \in H(\widetilde{L}_\bs)$ is given by
\begin{equation}\label{rep_int_rep}
h(\bx) = \int_{\R^{\N}} \widetilde{L}_\bs (\bx,\by) \, d \mu(\by), \quad \bx \in \R^\N.
\end{equation}

\begin{lemma}\label{l5i}
For every sequence $\bs$, with
$\sum_{j\in\N} \sigma^2_j = \infty$
we have 
$\int_{\R^\N} f \, d\mu = 0$
for all $f \in H(\widetilde{L}_\bs)$. 
\end{lemma}

\begin{proof}
We obtain for the norm of the representer $h$ of integration, cf. \eqref{rep_int_rep},
that
\begin{align*}
\|h\|^2_{H( \widetilde{L}_\bs)} &= 
\int_{\R^\N} \! \int_{\R^\N} \widetilde{L}_\bs (\bx,\by) \,
d\mu(\bx) \, d\mu(\by)\\
&=
\prod_{j \in \N}
\int_{\R} \! \int_{\R} \ell_{\sigma_j} (x_j,y_j) \,
d\mu_0(x_j) \, d\mu_0(y_j)\\
&=
\prod_{j \in \N}
\frac{1}{(1+4\sigma_j^2)^{1/2}}\\
&= 0,
\end{align*}
where the second last identity follows from \citet[Eqn.~(1.5)]{KSW2017} (or can be calculated directly). Consequently, we have $h(\bx) = 0$ for every $\bx \in \R^\N$ and the integration functional on $H(\widetilde{L}_\bs)$ is trivial.
\end{proof}

\begin{lemma}\label{l5ii}
For every sequences $\bs$ fulfilling \eqref{g1a}, the following holds true:
\begin{itemize}
\item[(i)]
We have $\mu(\ell^2(\bs^2)) = 1$.
\item[(ii)]
We have $\widetilde{L}_\bs(\by,\bz) = 0$ and $h(\bz) = 0$
for $\by \in \ell^2(\bs^2)$ and $\bz \in \R^J \setminus
\ell^2(\bs^2)$. 
\end{itemize}
\end{lemma}

\begin{proof}
We establish (i) and (ii) in the non-trivial case $d=\infty$.

Naturally, Lemma~\ref{l4} yields (i).


Now let $\by \in \ell^2(\bs^2)$
and $\bz \in \R^\N \setminus \ell^2(\bs^2)$.
Obviously, 
$\widetilde{L}_\bs(\by,\bz) = 0$. 
Furthermore, (i) and identity~\eqref{rep_int_rep} yield
for $\bz \in \R^\N \setminus \ell^2(\bs^2)$ that
\[
h(\bz) = \int_{\ell^2(\bs^2)} \widetilde{L}_\bs (\bz,\by) \, d \mu(\by) = 0.
\]

\end{proof}

Let now $\bs$ be a sequence that satisfies \eqref{g1a}.
For $d = \infty$ we prefer to 
replace the full Cartesian product $\R^\N$ in the definition of
$\widetilde{L}_\bs$ by the proper subset $\ell^2(\bs^2)$ of full measure.
Lemma~\ref{l5ii} reveals that this restriction 
has no impact on the results for integration and
$L^2$-approximation, 
while it simplifies the presentation substantially. 
See \citet[Sec.~4.2]{GHHR2021} for further results on different
domains for $\widetilde{L}_\bs$ in the case $d = \infty$. 

For every $d$ the mapping 
\[
L_\bs := \widetilde{L}_\bs|_{\ell^2(\bs^2) \times \ell^2(\bs^2)}
\]
and the Hilbert space $H(L_\bs)$ are called
a Gaussian kernel and, despite a different use in stochastic analysis,
a Gaussian space, respectively.
Accordingly, $H(L_\bs)$ consists of
real-valued functions on the domain 
\[
\X(L_\bs) := \ell^2(\bs^2).
\]
\subsection{Tensor Products of Hermite Kernels}

For the study of tensor products of Hermite kernels we consider a
sequence $\bb := (\beta_j)_{j \in J}$ of base parameters, i.e., $0
< \beta_j < 1$. Observe that 
\begin{equation}\label{g46}
k_{\beta_j}(x,x)\geq 1
\end{equation}
for all $j\in J$ and $x\in\R$, since $h_0=1$.

For every sequence $\bb$ with
\begin{equation}\label{g1b}
\sum_{j \in J} \beta_j < \infty
\end{equation}
we define the Hermite kernel $K_\bb$ by
\[
K_\bb (\bx,\by) := \prod_{j \in J} k_{\beta_j} (x_j,y_j)
\]
for $\bx, \by \in \X(K_\bb)$ with the maximal domain 
\[
\X(K_\bb) := \{ \bx \in \R^J \colon \prod_{j \in J} k_{\beta_j} (x_j,x_j)
< \infty \}.
\]
If $d < \infty$ then \eqref{g1b} is trivially satisfied and 
$\X(K_\bb) = \R^d$.

\begin{lemma}\label{l6}
Given~\eqref{g1b}, we have
\begin{itemize}
\item[(i)]
$\X(K_\bb) = \ell^2(\bb)$ and $\mu(\ell^2(\bb)) = 1$,
\item[(ii)]
$H(K_\bb) \subseteq L^2(\mu)$ with a compact identical embedding. 
\end{itemize}
\end{lemma}

\begin{proof}
Assume that~\eqref{g1b} is satisfied.
We establish (i) in the non-trivial case $d=\infty$.
See \citet[Prop.\ 3.18]{GHHR2021} for a general result that
yields $\X(K_\bb) = \ell^2(\bb)$. Lemma~\ref{l4} implies 
$\mu(\ell^2(\bb)) = 1$,
see also \citet[Prop.\ 3.10]{GHHR2021} for a general result.

By orthogonality of the Hermite polynomials we have
\[
\int_{\R} k_{\beta_j}(x_j,x_j) \, d \mu_0(x_j) = 
\sum_{\nu \in \N_0} \beta_j^\nu = (1-\beta_j)^{-1}.
\]
Hence we may use \eqref{g46} and \eqref{g1b} to obtain
\[
\int_{\ell^2(\bb)} K_\bb (\bx,\bx) \, d \mu(\bx)
= 
\int_{\R^J} \prod_{j \in J} k_{\beta_j}(x_j,x_j)
\, d \mu(\bx) 
= 
\prod_{j \in J} (1-\beta_j)^{-1} < \infty,
\]
i.e., the kernel $K_\bb$ has finite trace. The latter
implies (ii), see, e.g., \citet[Lemma~2.3]{StSc2012}.
\end{proof}

The Hilbert space $H(K_\bb)$, which due to Lemma~\ref{l6} consists of 
real-valued functions on the domain
\[
\X(K_\bb) = \ell^2(\bb),
\]  
is called a Hermite space.
This domain is a the proper subset of $\R^\N$ of full measure
in the case $d = \infty$ and the full Cartesian product $\R^d$ in the case
$d< \infty$.

\subsection{The Isometric Isomorphism}\label{SEC:Isom_Isom}

As we will see, the relation
\begin{equation}\label{g2}
1 - \beta_j = \frac{2}{1+ \left(1+8 \sigma_j^2 \right)^{1/2}}
\end{equation}
for every $j \in J$ is of key importance. First of all, 
we note that \eqref{g2} defines a bijection between the set 
of shape parameters $\sigma_j > 0$ and the set of 
base parameters $0 < \beta_j < 1$. 

\begin{rem}\label{r4}
Consider the case $d = \infty$,
and assume that \eqref{g2} is satisfied for every $j \in \N$.
If $\lim_{j \to \infty} \sigma_j = 0$ or $\lim_{j \to \infty}
\beta_j = 0$, then $\beta_j \asymp \sigma_j^2$ and therefore
$\ell^2(\bs^2) = \ell^2(\bb)$ as vector spaces. Moreover, 
\eqref{g1a} and \eqref{g1b} are equivalent. We conclude that
\eqref{g2} defines a bijection between the set of square-summable
sequences of shape parameters $\sigma_j > 0$ and the set of summable
sequences of base parameters $0 < \beta_j < 1$. 
\end{rem}

Let $\bc := (c_j)_{j \in J}$ be a positive sequence with
$\sum_{j \in J} \omega_j < \infty$ for $\omega_j := |c_j-1|$. 
We define $\phi_\bc \colon \R^J \to {[0,\infty[}$ by
\[
\phi_\bc (\bx) :=
\begin{cases}
\exp\Bigl(-\sum_{j \in J} \frac{c_j^2-1}{4} \cdot x_j^2 \Bigr) &
\text{if $\bx \in \ell^2(\bo)$,} \\
0 & \text{otherwise.}
\end{cases}
\]
Note that 
$\bx \mapsto \bc \bx$ and $\bx \mapsto \bc^{-1} \bx$
define bijections on $\ell^2(\bs^2)$. 
We define a linear mapping $Q_\bc$ on 
the space of all functions $f \colon \ell^2(\bs^2) \to \R$ by
\[
Q_\bc f (\bx) := \bc_\ast^{1/2} \cdot \phi_\bc(\bx) \cdot f(\bc \bx)
\]
for $\bx \in \ell^2(\bs^2)$; cf.\ \eqref{g44} for the
definition of $\bc_\ast$.

See \citet[Thm.~5.8]{GHHR2021} for the following result%
\footnote{Instead of $k_{\beta_j}$ and $K_\bb$,
the kernels $(1-\beta_j) k_{\beta_j}$ and their tensor product are 
studied in \citet{GHHR2021}.
}.

\begin{theo}\label{t1}
Assume that \eqref{g1a} is satisfied. Moreover,
assume that \eqref{g2} and
\begin{equation}\label{g3}
c_j = \left( 1 + 8 \sigma_j^2 \right)^{1/4}
\end{equation}
are satisfied for every $j \in J$. 
Then $Q_\bc$ defines an isometric isomorphism 
on $L^2(\mu)$ and between $H((\be-\bb)_\ast K_\bb)$ and $H(L_\bs)$.
\end{theo}

By definition, the value of $Q_\bc f$ at $\bx$ is determined by
the value of $f$ at $\bc \bx$. The analogous property for
$Q_\bc^{-1}$ reads as follows.

\begin{lemma}\label{l1}
Assume that \eqref{g1a} is satisfied. Moreover,
assume that \eqref{g3} is satisfied for every $j \in J$.
Then we have $\phi_\bc > 0$ on $\ell^2(\bs^2)$ and
\[
Q_\bc^{-1} f(\bx) = 
\frac{1}{\bc_\ast^{1/2} \phi_\bc(\bc^{-1}\bx)} \cdot f(\bc^{-1}\bx)
\]
for all $f \colon \ell^2(\bs^2) \to \R$ and $\bx \in \ell^2(\bs^2)$.
\end{lemma}

\begin{proof}
We only have to verify that $\phi_\bc > 0$ on $\ell^2(\bs^2)$.
In the non-trivial case $d=\infty$
the latter follows from $c_j^2 - 1 \asymp \sigma_j^2$,
which holds due to \eqref{g1a} and \eqref{g3}.
\end{proof}

\section{The Worst-Case Setting}\label{s2a}

$L^2$-approximation and integration on Gaussian and on Hermite spaces 
will be studied within the following framework,
where $\X := \ell^2(\bs^2)$ and $M := L_\bs$ under the
assumption \eqref{g1a} or $\X := \ell^2(\bb)$ and $M := K_\bb$ under the
assumption \eqref{g1b}. Recall that $\mu(\X) = 1$ and $H(M)
\subseteq L^2(\mu)$ in both cases, see Lemmata~\ref{l5},  \ref{l5ii}.(i), and \ref{l6}.

For $L^2$-approximation of functions $f \in H(M)$
we consider linear sampling methods 
\begin{equation}\label{g5}
A(f) := \sum_{i=1}^n f (\bx_i) \cdot a_i
\end{equation}
with $n \in \N$, nodes $\bx_i \in \X$, and coefficients
$a_i \in L^2(\mu)$.
As error criterion we consider the worst-case error of $A$ on the 
unit ball in $H(M)$, which is defined by
\[
e(A,M) := \sup_{\|f\|_{H(M)} \leq 1} \| f - A(f)\|_{L^2(\mu)}.
\]
For integration of functions $f \in H(M)$ 
we consider quadrature formulas 
\begin{equation}\label{g6}
A(f) := \sum_{i=1}^n f (\bx_i) \cdot a_i
\end{equation}
with $n$ and $\bx_i$ as before, but with
coefficients $a_i \in \R$.
The worst-case error of $A$ on the unit ball in $H(M)$ is defined by
\[
e(A,M) := \sup_{\|f\|_{H(M)} \leq 1} | I (f) - A(f) |,
\]
where
\[
I(f) := \int_{\X} f \, d \mu.
\]

The worst-case cost of an algorithm $A$ of the form \eqref{g5} or 
\eqref{g6} is defined by
\begin{equation}\label{eq:worst_case_cost}
\cost (A) := \sum_{i=1}^n \cost(\bx_i),
\end{equation}
where $\cost \colon \X \to {[0,\infty[} \cup \{\infty\}$ denotes
the cost of a single function evaluation,

Of course, the most natural choice of the cost function is 
$\cost(\bx) := 1$ for every $\bx \in \X$, which means that the 
functions from $H(M)$ may be evaluated anywhere in the domain $\X$ 
at a constant cost one. We will use this model throughout the paper
in the case $d \in \N$, where $\X = \R^d$.
In the case $d=\infty$, where $J=\N$ and $\X$ is an
infinite-dimensional space,
we employ the unrestricted subspace 
sampling model, which has been introduced in \citet{KuoEtAl10}. 
This model is based on a non-decreasing cost function 
$\$\colon \N_0 \to \left[1,\infty\right[$ 
 in the following way. For $\bx \in \R^\N$ the 
number of active variables is defined by
\[
\Act(\bx) :=\#\{j\in\N \colon x_j \neq 0\},
\]
and the evaluation of any function $f \in H(M)$ is permitted at
any point $\bx$ for which $\Act(\bx)$ is finite; in this case 
the corresponding cost is given by $\$ (\Act(\bx))$. 
We add that $\Act(\bx) <\infty$ already implies $\bx\in\X$.

Accordingly, we define
\begin{equation}\label{eq:ev_cost}
\cost(\bx) := 
\begin{cases}
\$(\Act(\bx)) & \text{if $\Act(\bx)<\infty$,}\\
\infty & \text{otherwise}
\end{cases}
\end{equation}
for $x \in \X$.

As the key quantity, we study the 
$n$-th minimal worst-case error, which is defined by
\begin{equation}\label{eq:nth_min_error}
e_n(M) := \inf \{e(A,M) \colon \cost(A) \leq n\}
\end{equation}
for $n \in \N$. For $n=0$ the corresponding quantity is
\[
e_0(M) := e(A,M) 
\]
with $A:=0$, i.e.,
\[
e_0(M) = \sup_{\|f\|_{H(M)} \leq 1} \|f\|_{L^2(\mu)} 
\]
for $L^2$-approximation and
\[
e_0(M) = \sup_{\|f\|_{H(M)} \leq 1} |I(f)|
\]
for integration.

For $d \in \N$ explicit formulas are available%
for the norm of the embedding of $H(M)$ into $L^2(\mu)$,
see \citet{SW2018} and \citet{IKPW16b},
and for the norm of the integration functional on $H(M)$,
see \citet{KSW2017} and \citet{IKLP15}.
In the following lemma we cover the case $d = \infty$, too.

\begin{lemma}\label{l7}
If \eqref{g1a} is satisfied then 
\[
e_0(L_\bs) = \prod_{j \in J} 
\frac{\sqrt{2}}{\left(1+ \left(1+8 \sigma_j^2 \right)^{1/2}\right)^{1/2}}
\]
for $L^2$-approximation and
\[
e_0(L_\bs) = \prod_{j \in J} \frac{1}{\left(1+ 4\sigma_j^2\right)^{1/4}}
\]
for integration.
If \eqref{g1b} is satisfied then
\[
e_0(K_\bb) = 1
\]
for $L^2$-approximation and for integration.
\end{lemma}

\begin{proof}
We use the tensor product structure of the
embedding of $H(M)$ into $L^2(\mu)$ and of the integration
functional on $H(M)$, see
\citet[Sec.~A.4--A.6]{GHHR2021} for a general account in the case
$d=\infty$, together with the known results for $d=1$. First of all,
\[
e^2_0(\ell_{\sigma_j})= 
\frac{2}{1+ \left(1+8 \sigma_j^2 \right)^{1/2}} < 1
\]
for $L^2$-ap\-prox\-i\-mation, 
see, e.g., \citet[Eqn.~(10)]{SW2018},
and
\[
e^4_0(\ell_{\sigma_j})= \frac{1}{1+ 4\sigma_j^2} < 1
\]
for integration,
see \citet[Eqn.~(1.5)]{KSW2017},
while $e_0(k_{\beta_j})=1$ for $L^2$-ap\-prox\-i\-mation and for
integration,
see \citet[p.~104]{IKPW16b} and \citet[p.~385]{IKLP15},
respectively.
Consequently,
\[
e_0(M) = \prod_{j \in J} e_0(m_j)
\]
with $m_j := \ell_{\sigma_j}$ or with $m_j := k_{\beta_j}$.
\end{proof}

\begin{rem}\label{rem:norm_zero}
Consider the reproducing kernel $\widetilde{L}_\bs$ 
on the domain $\R^\N$, as in \eqref{eq:kernel_full_domain}. If \eqref{g1a} is not satisfied, we still have $H(\widetilde{L}_\bs) \subseteq L^2(\mu)$ due to Lemma~\ref{l5}. However, in this case, integration is trivial, see Lemma~\ref{l5i}.
Similarly, one can show that 
the $L^2$-approximation operator is trivial.
\end{rem}

It is well-known that in the worst-case setting
non-adaptive linear
algorithms of the form \eqref{g5} for $L^2$-approximation
and \eqref{g6} for integration
are optimal within the class of 
all deterministic algorithms 
based on function values and with information cost at most $n$;
for more details see \citet{TrWaWo88}.

\section{Integration}\label{s4}

In what follows, we derive new results for integration and 
$L^2$-approximation on Gaussian and Hermite spaces. This will be done by 
transferring known results from Hermite spaces with the help of the 
isometric isomorphism from Section~\ref{SEC:Isom_Isom} to Gaussian 
spaces and vice versa. Since the actual transference mechanisms 
for integration and $L^2$-approximation 
differ significantly, we discuss both applications separately in 
Sections~\ref{s4} and \ref{s3}, respectively.

Throughout this section, $\bs := (\sigma_j)_{j\in J}$ and
$\bb := (\beta)_{j\in J}$ denote
sequences of shape parameters and base parameters, respectively.
Initially, we do not impose any summability requirements on 
$\bs$ or $\bb$. For the integration problem, an appropriate 
relation between $\sigma_j$ and $\beta_j$, which is
different from \eqref{g2}, will be
determined later, see \eqref{g20}.

\subsection{The Transference Result}\label{subsec:transference_int}

For any positive sequence $\bt \in \R^J$ we define 
$t_\bt \colon \R^J \to \R^J$ by
$t_\bt (\bx) := \bt^{-1} \bx$.

\begin{lemma}\label{l2}
Let $\bt := (\tau_j)_{j \in J}$ denote a positive 
sequence with $\sum_{j \in J} |\tau_j-1| < \infty$. 
The image measure $t_\bt \mu$ of $\mu$
with respect to $t_\bt$ has the density
$\bt_\ast \cdot \phi_\bt^2$
with respect to $\mu$. 
\end{lemma}

\begin{proof}
To establish the statement of the lemma, it suffices to show for an 
arbitrary cylinder set $A \subseteq \R^J$ that 
$t_\bt \mu(A) = \int_A \bt_\ast \cdot \phi_\bt^2\, d\mu$. 

In the case $d < \infty$ any cylinder set is of the form 
$A := A_1\times\dots\times A_d$ 
with measurable sets $A_j \subseteq \R$. Since
\begin{align*}
\mu_0(\tau_j A_j) &= 
(2\pi)^{-1/2} \cdot \int_{\tau_j A_j} \exp(-x_j^2/2)\, d x_j \\
&=
\tau_j \cdot \int_{A_j}
\exp\biggl(-\frac{\tau_j^2-1}{2} \cdot x_j^2 \biggr) \, d\mu_0 (x_j), 
\end{align*}
we obtain
\[
t_\bt\mu(A) = \mu \left( t^{-1}_\bt (A) \right) = 
\prod_{j=1}^{d} \mu_0 ( \tau_j A_j)
= \bt_\ast \cdot \int_{A} \phi_\bt^2(\bx) \, d\mu(\bx),
\]
which in turn yields the claim.

In the case $d = \infty$ 
we put $J_0 := \{ j \in \N \colon \tau_j < 1\}$.
Using the monotone convergence theorem we obtain
\begin{equation}\label{ineq:mono_conv_theo}
\int_{\R^\N} \phi_\bt^2 (\bx) \, d \mu(\bx)
\leq
\int_{\R^\N} 
\exp \Biggl( - \sum_{j \in J_0} \frac{\tau_j^2-1}{2} \cdot x_j^2 \Biggr) 
\, d \mu(\bx)
= \prod_{j \in J_0} \tau_j^{-1} < \infty.
\end{equation}
For $d=\infty$ 
any cylinder set is of the form $A := A_1\times\dots$
with measurable sets $A_j \subseteq \R$, where $A_j = \R$ for $j >
j_0$. For every $j_1 > j_0$ we obtain
\begin{align*}
t_\bt\mu(A) &= 
\prod_{j=1}^{j_1} \mu_0 ( \tau_j A_j)
= \prod_{j=1}^{j_1} \biggl( \tau_j \cdot \int_{A_j} 
\exp\biggl(-\frac{\tau_j^2-1}{2} \cdot x_j^2 \biggr)
\, d\mu_0(x_j) \biggr) \\
&= \prod_{j=1}^{j_1} \tau_j \cdot \int_{A} 
\exp\Biggl(-\sum_{j=1}^{j_1} \frac{\tau_j^2-1}{2} \cdot x_j^2 
\Biggr) \, d\mu(\bx).
\end{align*}
Clearly, we have
\[
\exp\Biggl(-\sum_{j=1}^{j_1} \frac{\tau_j^2-1}{2} \cdot x_j^2 
\Biggr) 
\le \exp 
\Biggl( - \sum_{j \in J_0} \frac{\tau_j^2-1}{2} \cdot x_j^2 \Biggr)  
\]
for all $\bx\in \R^\N$. Thus, due to \eqref{ineq:mono_conv_theo}, 
the dominated convergence theorem yields
\[
t_\bt\mu(A) = \bt_\ast \cdot \int_A \phi_\bt^2(\bx) \, d\mu(\bx),
\]
which in turn establishes the claim. 
\end{proof}

\begin{lemma}\label{l3}
Let $\bc := (c_j)_{j \in J}$ and 
$\bt := (\tau_j)_{j \in J}$ denote positive 
sequences with $\sum_{j \in J} |c_j-1| < \infty$ and
$\sum_{j \in J} |\tau_j-1| < \infty$. 
If
\begin{equation}\label{g9b}
\tau_j^2 = \frac{c_j^2+1}{2} 
\end{equation}
for every $j \in J$ then we have
\begin{equation}\label{g9a}
I(Q_\bc f \circ t_\bt) = \frac{\bt_\ast}{\bc_\ast^{1/2}} \cdot I(f) 
\end{equation}
for every $f \in L^1(\mu)$. 
\end{lemma}

\begin{proof}
By definition,
\[
Q_\bc f \circ t_\bt  = 
\bc_\ast^{1/2} \cdot \left(\phi_\bc \circ t_\bt \right) \cdot
\left(f \circ t_{\bc^{-1} \bt} \right) =
\left( \bc_\ast^{1/2} \cdot \left(\phi_\bc \circ t_\bc \right) 
\cdot f \right)
\circ t_{\bc^{-1} \bt}. 
\]
Note that $\sum_{j \in J} \omega_j < \infty$ for
$\omega_j := |c_j^{-1} \tau_j-1|$.
Lemma~\ref{l2} implies
\[
I(Q_\bc f \circ t_\bt) =
\frac{\bt_\ast}{\bc_\ast^{1/2}} \cdot \int_{\R^J}
\left(\phi_\bc \circ t_\bc\right) \cdot \phi^2_{\bc^{-1}\bt} \cdot 
f \, d \mu 
\]
for every $f \colon \R^J \to \R$ such that 
$\left(\phi_\bc \circ t_\bc\right) \cdot \phi^2_{\bc^{-1}\bt} \cdot f \in
L^1(\mu)$. 

Lemma~\ref{l4} yields $\mu(\ell^2(\bo)) = 1$.
Let $\bx \in \ell^2(\bo)$. We have $\phi_{\bc^{-1} \bt}(\bx) >
0$ and, since $\omega_j \asymp |c_j -1|$, we also have
$\phi_\bc(t_\bc(\bx)) > 0$.
Since 
\[
- \ln \left(
\phi_\bc (t_\bc (\bx)) \cdot \phi^2_{\bc^{-1}\bt}(\bx)\right) 
=
\sum_{j \in J}
\frac{x_j^2}{4c_j^2} \cdot \left( c_j^2-1 + 2(\tau_j^2-c_j^2) \right),
\]
we conclude that \eqref{g9a} is satisfied for every $f \in
L^1(\mu)$ if $c_j^2-1 = 2(c_j^2 - \tau_j^2)$ for every $j \in J$.
The latter is equivalent to \eqref{g9b} for every $j \in J$.
\end{proof}

\begin{theo}\label{t3}
Assume that \eqref{g1a} is satisfied.
Moreover, assume that
\begin{align}
1-\beta_j &= \frac{1}{1+2\sigma_j^2}, \label{g20}\\
c_j &= (1+4\sigma_j^2)^{1/2}, \label{g21}\\
\tau_j &= (1+2\sigma_j^2)^{1/2} \label{g22}
\end{align}
are satisfied for every $j \in J$.
For every quadrature formula $A$ with nodes from $\ell^2(\bs^2)$
and
\[
B(f) := \frac{\bc_\ast^{1/2}}{\bt_\ast} \cdot 
A (Q_\bc f \circ t_\bt)
\]
we have
\[
e(A,L_{\bs}) = (\be+4\bs^2)_\ast^{-1/4} \cdot e(B, K_\bb).
\]
\end{theo}

\begin{proof}
At first, we only assume that \eqref{g1a} is satisfied.
We alter the shape parameters of the Gaussian kernel
$L_\bs$. To this end, let $\bt := (\tau_j)_{j \in J}$ denote 
any positive sequence with $\sum_{j \in J} |\tau_j-1| < \infty$. 
Obviously, $\sum_{j \in J} (\tau_j \sigma_j)^2 < \infty$.
Recall that $t_\bt$ defines a bijection on $\ell^2(\bs^2)$.
Since 
\[
L_{\bt \bs}(\bx,\by) 
= \prod_{j \in J} \ell_{\tau_j \sigma_j}(x_j,y_j) 
= \prod_{j \in J} \ell_{\sigma_j}(\tau_j x_j,\tau_j y_j) 
= L_{\bs} (\bt \bx, \bt \by)
\]
for all $\bx,\by \in \ell^2(\bs^2)$, the spaces 
$H(L_{\bt \bs})$ and $H(L_{\bs})$ are isometrically isomorphic via 
$f \mapsto f \circ t_\bt$.
Consequently,
\[
e(A,L_\bs) = 
\sup_{\|f\|_{H(L_{\bt \bs})} \leq 1} 
|I(f \circ t_\bt) - A (f \circ t_\bt)|.
\]
	
Next, we apply Theorem~\ref{t1} with $\bt \bs$
instead of $\bs$, which leads to
\begin{equation}\label{g7}
1 - \beta_j = \frac{2}{1+ \left(1+8 (\tau_j \sigma_j)^2 \right)^{1/2}}
\end{equation}
and
\begin{equation}\label{g8}
c_j = \left(1+8 (\tau_j \sigma_j)^2\right)^{1/4},
\end{equation}
instead of \eqref{g2} and \eqref{g3}.
Assume that \eqref{g7} and \eqref{g8} are satisfied for every $j
\in J$.
Since $\sum_{j \in J} (\tau_j \sigma_j)^2 < \infty$, we obtain 
$\sum_{j \in J} \beta_j < \infty$ by Remark \ref{r4}, which in 
turn ensures that $K_\bb$ is well-defined.
 Let $Q := Q_\bc$. Then
$Q_{|H(K_\bb)}$ is an isomorphism from
$H(K_\bb)$ to $H(L_{\bt \bs})$ with
\[
\Vert Q f\|^2_{H(L_{\bt \bs})} =
\Vert f\|^2_{H((\be-\bb)_\ast K_\bb)} = (\be-\bb)_\ast^{-1} \cdot
\Vert f\|^2_{H(K_\bb)}
\]
for every $f \in H(K_\bb)$. We conclude that
\[
e(A,L_\bs) =
(1-\bb)_\ast^{1/2} \cdot \sup_{\|f\|_{H(K_\bb)} \leq 1} 
|I(Q f \circ t_\bt) - A (Q f \circ t_\bt)|.
\]

Finally, we apply Lemma~\ref{l3}, noting that
the unique solution of \eqref{g9b} and \eqref{g8} with 
$\tau_j > 0$ is given by \eqref{g21} and \eqref{g22}.
In the sequel, we assume that 
\eqref{g20}, \eqref{g21}, and \eqref{g22} are satisfied,
in addition to \eqref{g1a}.
We use 
\[
(\tau_j \sigma_j)^2 = (1+2\sigma_j^2) \sigma_j^2
\]
to obtain
\[
1+ 8 (\tau_j\sigma_j)^2 = (1+4\sigma_j^2)^2.
\]
It follows that \eqref{g7} is satisfied, too,
and obviously, $\sum_{j \in J} |\tau_j-1| < \infty$ as well as 
$\sum_{j \in J} |c_j-1| < \infty$. Lemma~\ref{l3} yields
\[
e(A,L_\bs)
=
(1-\bb)_\ast^{1/2} \cdot \frac{\bt_\ast}{\bc_\ast^{1/2}} \cdot
\sup_{\|f\|_{H(K_\bb)} \leq 1} 
\Biggl|I(f) - \frac{\bc_\ast^{1/2}}{\bt_\ast} \cdot A
(Q f \circ t_\bt)\Biggr|.
\]

We use
\[
\frac{\tau_j}{c_j^{1/2}} = 
\frac{\left(1+2\sigma_j^2\right)^{1/2}}{\left(1+4\sigma_j^2\right)^{1/4}},
\]
to obtain
\[
(1-\beta_j)^{1/2} \frac{\tau_j}{c_j^{1/2}} = 
\frac{1}{(1+4\sigma_j^2)^{1/4}}
\]
and therefore
\[
(1-\bb)_\ast^{1/2} \cdot \frac{\bt_\ast}{\bc_\ast^{1/2}} =
(\be+4\bs^2)_\ast^{-1/4}.
\qedhere
\]
\end{proof}

\begin{rem}\label{r7}
Let the assumptions from Theorem~\ref{t3} be satisfied.
For a quadrature formula $A$ on the Gaussian space $H(L_\bs)$
according to \eqref{g6} with nodes from $\ell^2(\bs^2)$
the quadrature formula $B$ on the Hermite space $H(K_\bb)$ according 
to Theorem \ref{t3} is easily determined explicitly. 
Let
\[
e_j := \Bigl(\frac{1+4\sigma_{j}^2}{1+2\sigma_j^2}\Bigr)^{1/2}
\]
for every $j\in J$.
Then, for $f\colon \ell^2(\bs^2)\to\R$ we have
\[
B f = \bee_\ast \cdot \sum_{i=1}^{n} f(\by_i)\cdot b_i,
\]
with $\by_i\in\ell^2(\bs^2)$ and $b_i\in\R$ given by
\[
\by_i := \bee \bx_i
\qquad \text{and}\qquad
b_i := \phi_{\bc}(\bt^{-1} \bx_i) \cdot a_i.
\]

It is easily verified that $(\bx_i,a_i)\mapsto(\by_i,b_i)$ defines a
bijection on the set $\ell^2(\bs^2)\times \R$ of pairs of
nodes and coefficients. Furthermore,
\[
\cost(A) = \cost(B).
\]
Indeed, this equation is trivially satisfied with $\cost(A)=n$
for $d \in \N$, and for $d=\infty$ it follows from
$\Act (\bx_i) = \Act(\by_i)$, which, in turn, is a consequence
of $e_j \neq 0$ for all $j \in J$.
\end{rem}

\begin{cor}\label{c2}
Assume that \eqref{g1a} and \eqref{g20} are satisfied.
For every $n \in \N$ we have
\[
\frac{e_n(L_{\bs})}{e_0(L_\bs)} = \frac{e_n(K_{\bb})}{e_0(K_\bb)} = 
e_n(K_\bb).
\]
i.e., the normalized $n$-th minimal errors for integration on Gaussian 
spaces and on the corresponding Hermite spaces, related by \eqref{g20}, 
coincide.
\end{cor}

\begin{proof}
We combine Theorem~\ref{t3} and Remark~\ref{r7} to obtain
\[
e_n(L_{\bs}) = (\be+4\bs^2)_\ast^{-1/4} \cdot e_n(K_\bb)
\]
for every $n\in \N_0$. 
Use Lemma \ref{l7} to establish the desired result.
\end{proof}

\subsection{Functions of a Single Variable}\label{s4.2}

Let us mention prior work concerning integration on $H(\ell_\sigma)$
and on $H(k_\beta)$, where explicit upper bounds for the worst-case 
error $e(A,M)$ of particular quadrature formulas $A$ and explicit 
lower bounds for $e_n(M)$ are obtained. Of course, these upper bounds
immediately yield upper bounds for minimal errors.
Gauss-Hermite rules are analyzed in \citet{IKLP15} on Hermite spaces and 
in \citet{KW2012} as well as \citet{KSW2017}
on Gaussian spaces. Moreover, scaled Gauss-Hermite rules are 
studied in \citet{KOG21} on Gaussian spaces.
Lower bounds for the $n$-minimal
errors have been established in \citet{IKLP15} for Hermite
spaces and in \citet{KSW2017} for Gaussian spaces.

In the following we combine these results in the best possible way
with the transference result established in Theorem~\ref{t3} and 
Corollary~\ref{c2}.

\begin{theo}\label{t4}
Let
\begin{align*}
C_1(\sigma) &:= 2^{-1} \cdot (1+4\sigma^2)^{-1/4},\\
C_2(\sigma) &:= \pi^{-1/4} \cdot (1+2\sigma^2)^{-1/2},\\
C(\beta) &:= \pi^{-1/4} \cdot (1-\beta^2)^{1/4}.
\end{align*}
For every shape parameter $\sigma > 0$ and every $n \in \N$
the $n$-th minimal error $e_n(\ell_\sigma)$
for integration on the Gaussian space $H(\ell_\sigma)$ satisfies
\[
C_1(\sigma) \cdot
\left( \frac{\sigma^2}{1+2\sigma^2} \right)^{2n} \cdot (n+1)^{-2}
\leq e_n (\ell_\sigma)
\leq
C_2(\sigma) \cdot
\left( \frac{2\sigma^2}{1+2\sigma^2} \right)^{n} \cdot n^{-1/4}.
\]
For every base parameter $0 < \beta < 1$ and every $n \in \N$ 
the $n$-th minimal error $e_n(k_\beta)$
for integration on the Hermite space $H(k_\beta)$ satisfies
\[
\frac{1}{2} \cdot \left(\frac{\beta}{2}\right)^{2n} \cdot (n+1)^{-2} 
\leq e_n(k_\beta) \leq
C(\beta) \cdot \beta^n \cdot n^{-1/4}.
\]
\end{theo}

\begin{proof}
The upper bound for $e_n(\ell_\sigma)$ follows immediately
from \citet[Thm.~2.5]{KOG21}, and the
lower bound for $e_n(k_\beta)$ is established in \citet[Thm.~2]{IKLP15}.

Assume that
\[
1-\beta = \frac{1}{1+2\sigma^2}.
\]
First of all, Corollary~\ref{c2} and Lemma~\ref{l7} yield
\[
e_n(\ell_\sigma) = (1+4\sigma^2)^{-1/4} \cdot e_n(k_\beta).
\]
Moreover, we have
\[
\beta = \frac{2 \sigma^2}{1+2\sigma^2}
\]
and
\[
1- \beta^2 = \frac{1+4\sigma^{2}}{(1+2\sigma^2)^2}.
\]
Consequently,
the upper bound for $e_n(k_\beta)$ follows from the upper bound
for $e_n(\ell_\sigma)$ and 
the lower bound for $e_n(\ell_\sigma)$ follows from the lower bound
for $e_n(k_\beta)$.
\end{proof}

The lower bound for $e_n(\ell_\sigma)$ from Theorem~\ref{t4}
 improves the lower bound from \citet[Thm.~4.1]{KSW2017}, 
which is super-exponentially small in $n$, while
the upper bound for $e_n(k_\beta)$ from Theorem~\ref{t4}
only slightly improves the upper bound from \citet[Prop.~1]{IKLP15}. 
Obviously, the upper and lower bounds from Theorem~\ref{t4} are not 
tight at all.

\subsection{Functions of Finitely Many Variables}

In this section we study the multivariate case $d\in\N$,
where upper bounds for the minimal errors $e_n(M)$ have been 
established in
\citet{KSW2017} and \citet{KOG21} for Gaussian spaces as well as
in \citet{IKLP15} and \citet{IKPW16b}
for Hermite spaces. 
In all three papers, the upper bounds are obtained by 
quadrature formulas
that are full tensor products of suitable univariate
quadrature formulas, which have already been mentioned in
Section~\ref{s4.2}.
For lower bounds and for the study of tractability concepts
we refer to \citet{KSW2017} and \citet{IKLP15}.

\subsubsection{Exponential Convergence}

At first, 
we consider the case of $d\in\N$ being fixed,
and thus \eqref{g1a} and \eqref{g1b} are trivially satisfied for any
choice of parameters $\sigma_j$ and $\beta_j$, respectively.

Let either $M := L_\bs$ or $M:=K_\bb$.
Motivated by Theorem~\ref{t4},
we study whether there exist 
constants $C_{1},C_{2},p >0$ fulfilling
\begin{equation}\label{eq:EXP}
e_n(M)\leq C_{1}\cdot \exp({-C_{2}\cdot n^{p}})
\end{equation}
for all $n\in\N$. In this case, 
we say that exponential convergence is achieved for the integration
problem on $H(M)$ and
\[
p^*(M) := 
\sup\{ p \geq 0 \colon 
\exists\, C_1, C_2 > 0\ \forall\, n \in \N \colon
e_n(M)\leq C_{1}\cdot \exp({-C_{2}\cdot n^{p}}) \}
\]
is the most relevant quantity.

Exponential convergence has been studied by \citet{KSW2017} and
\citet{KOG21} for Gaussian spaces, as well as by \citet{IKLP15} for 
Hermite spaces.
We present a moderate generalization of the result for Gaussian
spaces, together with a proof that employs the transference mechanism.

\begin{theo}\label{theo:exp}
Exponential convergence is achieved for the integration
problem on $H(L_\bs)$ with
\[
p = 1/d = p^*(L_\bs).
\]
\end{theo}

\begin{proof}
After a possibly necessary reordering of the univariate
kernels, the shape
parameters satisfy $\sigma_1 \geq \dots \geq \sigma_d$, which implies
$\beta_1 \geq \dots \geq \beta_d$ for the base parameters given by
\eqref{g20}. Thus we may employ 
\citet[Thm.~1.1]{IKLP15} in the specific case 
$\omega := 1/\mathrm{e}$ as well as $a_j := \ln(1/\beta_j)$ and
$b_j := 1$ for all $j\in\N$, which yields exponential
convergence on $H(K_\bb)$ with $p^*(K_\bb) = 1/d$.
Furthermore, \eqref{eq:EXP} is satisfied 
for $M = K_\bb$ in the extremal case $p = 1/d$,
see \citet[Thm.~4]{IKLP15}.

Corollary \ref{c2} implies that \eqref{eq:EXP} holds for 
$M = K_\bb$ if and only if it holds for 
$M = L_\bs$ with $C_{1}$ being replaced by 
$C_{1}\cdot e_0(L_\bs)$, but with $C_{2}$ and $p$ unchanged.
Thus, we obtain the desired result.
\end{proof}

Theorem~\ref{theo:exp} has been established in 
\citet[Thm.~1.1]{KSW2017} under the additional assumption
\[
\max_{1 \leq j \leq d} \sigma_j^2<1/2, 
\]
while the general case has been left open for future research.
In the case 
\[
\sigma_1 = \dots = \sigma_d > 0
\]
the exponential convergence on $H(L_\bs)$ with $p=1/d$ follows directly 
from \citet[Cor.~2.11]{KOG21}. 

Alternatively to the transference mechanism, the following
monotonicity property allows to establish Theorem~\ref{theo:exp} in
full generality and working exclusively with Gaussian spaces.
Suppose that $\bs^{(1)} \leq \bs^{(2)}$,
i.e., $\sigma^{(1)}_j \leq \sigma^{(2)}_j$ for $j=1,\dots,d$.
Using \eqref{g12} we conclude that $H(L_{\bs^{(1)}}) \subseteq
H(L_{\bs^{(2)}})$ with a continuous identical embedding.
Therefore we obtain $p^*(L_\bs) \leq 1/d$ for every $\bs$ from
\citet[Thm.~1.1.(b)]{KSW2017}, and exponential convergence with 
$p=1/d$ for every $\bs$ from \citet[Cor.~2.11]{KOG21}.

\subsubsection{Exponential Convergence Weak Tractability}

Exponential convergence has been studied alongside uniform 
exponential convergence and several notions of tractability in
\citet{IKLP15}, \citet{IKPW16b}, and \citet{KSW2017}. In this paper, 
the setting is specified by arbitrary sequences
$\bs:=(\sigma_j)_{j\in\N}$ and $\bb := (\beta_j)_{j\in\N}$ 
of shape parameters $\sigma_j > 0$ and base parameters
$0 < \beta_j < 1$. For every $d\in\N$ we put
$\bs_d := (\sigma_1,\dots,\sigma_d)$ and 
$\bb_d := (\beta_1,\dots,\beta_d)$, and we consider
either $M_d := L_{\bs_d}$ for all $d\in\N$ or $M_d := K_{\bb_d}$ 
for all $d\in\N$.

By definition, we have uniform exponential convergence
if there exists a constant $p>0$ with the following property: For every
$d \in \N$ there exist $C_1,C_2>0$ such that
\[
e_n(M_d)\leq C_{1}\cdot \exp({-C_{2}\cdot n^{p}})
\]
for all $n \in \N$. For all sequences $\bs$, Theorem~\ref{theo:exp} 
excludes uniform exponential convergence, and the lack of this property
excludes exponential convergence polynomial 
tractability, see \citet[p.~832]{KSW2017}.  We are thus led 
to the study of a weaker tractability concept, 
which is discussed in the following.

For $0<\eps<1$
the information complexity for the absolute and the normalized
error criterion are defined by
\[
\nabs(\eps,M_d) := 
\inf\{n\in\N\colon e_n(M_d)\leq\eps\}
\]
and
\[
\nnorm(\eps,M_d) := \inf\{n\in\N\colon
e_n(M_d)/e_0(M_d)\leq\eps\},
\]
respectively.
Lemma~\ref{l7} yields
$\nabs(\eps,M_d)\leq\nnorm(\eps,M_d)$, 
with equality if $M_d$ is a Hermite kernel.

Let $t\geq 1$ and $\kappa > 0$. By definition, exponential convergence
$(t,\kappa)$-weak tractability, for short EC-$(t,\kappa)$-WT,
holds for $(M_d)_{d \in \N}$ and the normalized error criterion, if
\[
\lim_{d+\eps^{-1}\to\infty}
\frac{\ln(\nnorm(\eps,M_d))}{d^t+(\ln(\eps^{-1}))^\kappa} = 0.
\]
For Hermite spaces, the case $t=\kappa=1$ is studied
in \citet{IKLP15} as well as \citet{IKPW16b},
while for Gaussian spaces, the case $t\geq 1$ and $\kappa\geq 1$
is studied in \citet{KSW2017}.

First, we give a necessary condition and a sufficient condition for 
EC-$(1,1)$-WT in the case of Gaussian spaces and the
normalized error criterion. 
For non-increasing shape parameters the two conditions coincide
and hence we even obtain a characterization
of this tractability property.

\begin{theo}\label{theo:(1,1)-WT}
Let $\bs$ be a sequence of shape parameters. 
If EC-$(1,1)$-WT holds for
$(L_{\bs_d})_{d\in\N}$ and the normalized error criterion, then we have 
\[
\inf_{j \in \N} \sigma_{j}=0.
\]
If $\bs$ is non-increasing and 
\[
\lim_{j\to\infty}\sigma_j=0,
\]
then EC-$(1,1)$-WT holds for
$(L_{\bs_d})_{d\in\N}$ and the normalized error criterion.

\end{theo}

\begin{proof}

For both statements, we make use of corresponding results for Hermite 
spaces, similarly to the proof of Theorem~\ref{theo:exp}. Therefore, 
let $\bb$ be the sequence of base parameters given by \eqref{g20}.

For the first statement, suppose that $\inf_{j \in \N}\sigma_j>0$
holds. Then, we also have $\inf_{j \in \N}\beta_j>0$, which means 
there exists a
constant sequence $\widetilde{\bb}$ of base parameters such that $\bb$
is bounded from below by $\widetilde{\bb}$. Thus, we may use
\citet[Thm.~1.3]{IKLP15} in the specific case $\omega:=1/e$
as well as $a_j:=\ln(1/\beta_j)$ and $b_j:=1$, which shows that 
EC-$(1,1)$-WT does not hold for $(K_{\widetilde{\bb}_d})_{d\in\N}$.
Further, for all $d\in\N$ we have identical embeddings of
$H(K_{\widetilde{\bb}_d})$ into $H(K_{\bb_d})$ of norm one, 
cf.\ \eqref{eq:embedding_hermite}, which implies
$e_n(K_{\tilde{\bb}_d})\leq e_n(K_{\bb_d})$ for all
$d,n\in\N$. Thus, EC-$(1,1)$-WT does not hold for $(K_{\bb_d})_{d\in\N}$.
Together with Corollary~\ref{c2}, this shows the claim.

To show the second statement, we assume that $\bs$ is non-increasing 
and converges to 0. Then, $\bb$ is also
non-increasing and converges to 0. Thus, by \citet[Thm.~3]{IKPW16b},
again in the specific case 
$\omega:=1/e$ as well as $a_j:=\ln(1/\beta_j)$ and $b_j:=1$,
we have that EC-$(1,1)$-WT holds for $H(K_{\bb_d})_{d\in\N}$.
Again by Corollary~\ref{c2}, this shows the claim.
\end{proof}

Theorem~\ref{theo:(1,1)-WT} improves the previously best known sufficient 
condition for EC-$(1,1)$-WT for
Gaussian spaces, namely, that the sequence of shape parameters
is non-increasing and converges to 0 exponentially fast, 
see \citet[Thm.~1.1(e)]{KSW2017}. On the other hand, the
necessity of $\bs$ converging to 0 was already known under the
additional assumptions that the boundedness condition
\begin{equation}\label{g23}
\sup_{j \in \N} \sigma_j^2 < 1/2 
\end{equation}
is satisfied and $\bs$ is non-increasing, see
\citet[Thm.~1.1(f)]{KSW2017}. We remark that the first result for Hermite
spaces we have used in the proof,
\citet[Thm.~3]{IKPW16b},  is fully constructive.
Therefore, via the transference principle,
the sufficient condition in Theorem~\ref{theo:(1,1)-WT}
is also obtained constructively.

Our next goal is to study EC-$(t,k)$-WT in the case $t>1$. 
As an intermediary step, we derive a new upper  bound for
$\ln(\nnorm(\eps,L_{\bs_d}))$ with an explicit dependence 
on $\varepsilon$, $d$, and $\bs$.
For $d \in \N$ we put
\[
h(d) := 
\begin{cases}
0 & \text{if $d \leq 2$},\\
\ln(\ln(d)) & \text{otherwise}.
\end{cases}
\]
\begin{theo}\label{t5}
There exists a constant $C > 0$ such that
\[
\ln(\nnorm(\eps,L_{\bs_d})) \leq C \cdot \left(
d \, h(d) + d \ln(\ln(\eps^{-1})) +
\sum_{j=1}^d \left(\ln (\sigma_j)\right)^+ \right)
\]
for every $d\in\N$, every $\eps \in {]0,1/3]}$, and every
sequence $\bs := (\sigma_j)_{j \in \N}$ of shape parameters.
\end{theo}

\begin{proof}
Put
\[
\zeta_j := \ln \left( 1 + \frac{1}{2\sigma_j^2} \right).
\]
From \citet[Thm.~2.10]{KOG21} we get 
\[
e_N( L_{\bs_d}) \leq
\sum_{j=1}^{d} (1+2\sigma_j^2)^{-1/2}
\cdot \prod_{\substack{i=1\\i\neq j}}^d (1+4\sigma_i^2)^{-1/4}
\cdot \exp( - n_j \zeta_j)
\]
for all $d\in\N$ and $(n_1,\dots,n_d) \in \N^d$ with $N
:= \prod_{j=1}^d n_j$. Together with Lemma~\ref{l7} this
implies
\[
\frac{e_N( L_{\bs_d})}{e_0( L_{\bs_d})}
\leq 
\sum_{j=1}^{d} (1+2\sigma_j^2)^{-1/2} 
\cdot (1+4\sigma_j^2)^{1/4} \cdot \exp(- n_j \zeta_j)
\leq \sum_{j=1}^d \exp(- n_j \zeta_j).
\]
Choosing
\[
n_j := 
\left\lceil\frac{\ln(d/\eps)} {\zeta_j}\right\rceil
\]
we obtain $\exp(-n_j \zeta_j )\leq \eps/d$ and therefore
$\nnorm(\eps,L_{\bs_d})\leq N$, i.e.,
\begin{equation}\label{g14}
n(\eps,d) :=
\nnorm(\eps,L_{\bs_d})\leq 
\prod_{j=1}^d
\left( \frac{\ln(d/\eps)}{\zeta_j} + 1 \right), 
\end{equation}
cf.\ \citet[Eqn.~(3.9)]{KSW2017}.

For $z \geq 0$ we define
\[
\phi(z) := \ln(1+z),
\]
and we put $\tau := \ln(d/\eps) \geq 0$ as well as
\[
J_1 := \{ j \in \{1,\dots,d\} \colon \tau \leq 2 \zeta_j\}
\qquad \text{and} \qquad
J_2 := \{1,\dots,d\} \setminus J_1.
\]
Clearly,
$\phi(z) \leq z$ for every $z \geq 0$ and  
$\phi(z) \leq 3 \ln(z)$ if $z > 2$,
and therefore
\begin{align*}
\ln(n(\eps,d)) &\leq 
\sum_{j=1}^d \ln(\tau/\zeta_j + 1 )
\leq
\sum_{j \in J_1} \tau/\zeta_j + 3 \sum_{j \in J_2} \ln(\tau/\zeta_j)\\
& \leq 2 d +
3 \sum_{j \in J_2} \ln(\tau/\zeta_j).
\end{align*}

In the sequel, we assume that $\eps \in {]0,1/3]}$, so
that $\ln(\tau) \geq \ln(\ln(3)) > 0$. With $c_1 := 2/\ln(\ln(3))+3$
we obtain
\begin{equation}\label{g13}
\ln(n(\eps,d)) \leq c_1 d \ln(\tau)
+ 3 \sum_{j \in J_2} \ln(1/\zeta_j).
\end{equation}

Let us consider the particular case that 
\[
\inf_{j \in \N} \sigma_j \geq 2.
\]
We put
\[
\psi(y) := \ln \left( 1 + 1/(2y^2) \right)
\]
for $y \geq 2$, so that $\zeta_j = \psi(\sigma_j)$. Since
$\lim_{y \to \infty} y^2 \cdot \psi(y) = 1/2$, we have 
\[
c_2 := \inf_{y \geq 2} y^2 \cdot \psi(y) \in {]0,1/2]}.
\]
With $c_3 := 2 + \ln(1/c_2)/\ln(2)$ we obtain
\[
\ln \left( 1/ \psi(y)\right) \leq \ln(y^2/c_2)
= \ln(y) \cdot \left( 2 + \ln(1/c_2)/\ln(y) \right) \leq c_3 \ln(y) 
\]
for every $y \geq 2$. Consequently,
\[
\sum_{j \in J_2} \ln(1/\zeta_j) =
\sum_{j \in J_2} \ln\left((1/\psi(\sigma_j) \right)
\leq c_3 \sum_{j=1}^d \ln(\sigma_j).
\
\]
Together with \eqref{g13}, this implies
\[
\ln(n(\eps,d)) \leq c_1 d \ln(\tau) + 3 c_3 \sum_{j=1}^d \ln (\sigma_j).
\]

Observe that the upper bound in \eqref{g14} is a monotonically
increasing function in each of the variables $\sigma_j$.
For any sequence of shape parameters we therefore have
\[
\ln(n(\eps,d)) \leq c_1 d \ln(\tau) + 3 c_3 \sum_{j=1}^d 
\ln (\max(\sigma_j,2)).
\]
Since
\[
\ln (\max(\sigma_j,2)) \leq (\ln(\sigma_j))^+ + 1,
\]
we obtain
\[
\ln(n(\eps,d)) \leq c_1 d \ln(\tau) + 3 c_3 d + 3 c_3 \sum_{j=1}^d
(\ln(\sigma_j))^+.
\]

Finally, there exists a constant $c_4 > 0$ such that
\begin{align*}
\ln(\tau) &= 
\ln \left( \ln(d) + \ln(\eps^{-1}) \right) 
\leq
\ln \left( 2 \max (\ln(d),\ln(\eps^{-1})) \right) \\
&\leq
c_4 \ln \left( \max (\ln(d),\ln(\eps^{-1})) \right)  \leq 
c_4 \left( h(d) + \ln(\ln(\eps^{-1})) \right) 
\end{align*}
for all $d\in\N$ and $\eps \in {]0,1/3]}$.
\end{proof}

Next, we present a sufficient condition for EC-$(t,k)$-WT to hold 
for some $t>1$ and every $\kappa> 0$.

\begin{theo}\label{t6}
If there exist $c,\alpha \geq 0$ such that the shape parameters satisfy
\[
\sigma_j\leq \exp(c j^\alpha)
\]
for all $j\in\N$, then we have EC-$(t,\kappa)$-WT for
$(L_{\bs_d})_{d \in \N}$ and the normalized and the absolute
error criterion for every $t>\alpha+1$ and every $\kappa >0$.
\end{theo}

\begin{proof}
It suffices to consider the normalized error criterion. We apply the upper bound for
\[
n(\eps,d) := \nnorm(\min(\eps,1/3),L_{\bs_d}) \geq
\nnorm(\eps,L_{\bs_d}).
\]
from Theorem \ref{t5}. Clearly, we have
\[
\lim_{d+\eps^{-1}\to\infty} 
\frac{d \, h(d) + d\ln(\ln(\max(\eps^{-1},3)))}
{d^t+\ln(\eps^{-1})^{\kappa}}=0
\]
even for $t>1$, so it suffices to show
\[
\lim_{d+\eps^{-1}\to\infty} 
\frac{\sum_{j=1}^d \left(\ln (\sigma_j)\right)^+}
{d^t+ \ln(\eps^{-1})^{\kappa}}=0
\]
for $t > \alpha+1$. This is indeed true, since
\[
\sum_{j=1}^d \left(\ln (\sigma_j)\right)^+ 
\leq \sum_{j=1}^d c j^\alpha
\leq c d^{\alpha+1}. \qedhere
\]
\end{proof}

Theorem \ref{t6} extends the result established in 
\citet[Thm.~1.1.(h)]{KSW2017},
where the boundedness condition \eqref{g23}
is shown to imply EC-$(t,\kappa)$-WT for every $t>1$ and
every $\kappa \geq 1$ under the additional assumption
that $\bs$ is non-increasing.
Theorem~\ref{t6} yields, in particular, that the latter already holds if
$(\sigma_j)_{j \in \N}$ is polynomially bounded.

Next, we apply the transference result to obtain new
results for Hermite spaces, where the normalized and the
absolute error criterion coincide.

At first, corresponding to Theorem~\ref{t5},
we provide a new upper bound for $\ln(\nnorm(\varepsilon,K_{\bb_d}))$
with an explicit dependence on $\varepsilon$, $d$, and $\bb$.

\begin{theo}\label{t7}
There exists a constant $C > 0$ such that
\[
\ln(\nnorm(\eps,K_{\bb_d})) \leq C \cdot \left(
d \, h(d) + d \ln(\ln(\eps^{-1})) - 
\sum_{j=1}^d \ln (1-\beta_j) \right)
\]
for every $d\in\N$, every $\eps \in {]0,1/3]}$, and every
sequence $\bb := (\beta_j)_{j \in \N}$ of base parameters in $(0,1)$.
\end{theo}

\begin{proof}
Since \eqref{g20} implies $2 \sigma_j^2 < 1/(1-\beta_j)$, the
statement follows from Corollary~\ref{c2} and Theorem~\ref{t5}.
\end{proof}

Next, corresponding to Theorem~\ref{t6}, 
we give a sufficient condition for EC-$(t,\kappa)-WT$ to hold
for some $t>1$ and every $\kappa> 0$ in the case of Hermite spaces.

\begin{theo}\label{t8}
If there exist $c,\alpha \geq 0$
such that the base parameters satisfy
\[
\beta_j \leq 1 - \exp(- c j^\alpha)
\]
for all $j\in\N$, then we have EC-$(t,\kappa)$-WT for
$(K_{\bb_d})_{d \in \N}$ and the normalized error criterion
for every $t>\alpha+1$ and every $\kappa \geq 1$.
\end{theo}

\begin{proof}
The assumption on the shape parameters $\beta_j$ implies
$-\ln(1-\beta_j) \leq c j^\alpha$. Thus we may proceed as in the 
proof of Theorem~\ref{t6} to verify the claim.
\end{proof}

\begin{rem}\label{rem:ECWT}
On the one hand, rather mild assumptions on the shape parameters or 
the base parameters imply EC-$(t,\kappa)$-WT for every $t > 1$
regardless of choice of $\kappa$, see
Theorems~\ref{t6} and \ref{t8}. On the other hand, in the case
$t=1$, much stronger
properties are required even for EC-$(t,1)$-WT to hold, see
Theorem~\ref{theo:(1,1)-WT} and \citet[Thm.~1.3]{IKLP15}.
We add that necessary conditions for EC-$(t,\kappa)$-WT are only known 
in the case $t=1$ and $0<\kappa\leq1$.

\end{rem}

\subsection{Functions of Infinitely Many Variables}
\label{SUBSEC:integr_IMV}
In this section, we consider the case $d=\infty$.
For $M = L_\bs$ or $M=K_\bb$, we study the decay
\begin{equation}\label{eq:dec}
\dec(M) := \decay \left( (e_n(M))_{n \in \N} \right)
\end{equation}
of the $n$-th minimal worst-case errors, see \eqref{g4}.
Note that a lower bound for $\dec (M)$ corresponds to an upper bound for 
the $n$-th minimal errors $e_n(M)$ and vice versa.
Our goal is to determine $\dec(L_\bs)$, and to this end we put  
\begin{equation}\label{eq:rho}
\rho:=   \liminf_{j\to \infty}
\frac{\ln(1/\sigma_j^2)}{\ln(j)}.
\end{equation}
Note that \eqref{g1a} implies $\rho \geq 1$, see, e.g.,
\citet[Lemma~B.3]{GHHRW2020}.

Recall that $\$(m)$ denotes the cost of evaluating a function
$f \colon \ell^2(\bs^2) \to \R$ at any point $\bx$ with exactly $m$
non-zero components, see Section~\ref{s2a}.
We assume that 
there exist $c_1,c_2 > 0$ such that
\begin{equation}\label{g82}
c_1 \cdot m \leq \$(m) \leq \exp(c_2 \cdot m)
\end{equation}
for all $m \in \N$.

\begin{theo}\label{Thm:InfDimInt}
Assume that \eqref{g1a} is satisfied.
The polynomial decay rate of $n$-th minimal errors of integration on the 
Gaussian space $H(L_\bs)$ is 
\[
\dec(L_\bs) = \frac{1}{2} (\rho - 1).
\]
\end{theo}

\begin{proof}
Consider the sequence $\bb$ of base parameters given by \eqref{g20}. Corollary \ref{c2} immediately yields
\begin{equation}\label{eq:decay}
\dec(L_\bs) = \dec(K_\bb)
\end{equation}
and thus it suffices to show $\dec(K_\bb) = \frac{1}{2} (\rho - 1)$.
This has already been established in \citet{GneEtAl24}
for an arbitrary sequence $\bb$ of base parameters such that 
$\sum_{j \in \N} \beta_j < \infty$.
Indeed, put $r_j:= - \log_2(\beta_j)$.
After a possibly necessary reordering of the kernels, we have
$r_1=\inf_{j\in\N} r_j$. Then the weights 
$\beta_j^{-\nu} = 2^{r_j \cdot \nu}$ with $\nu,j\in\N$ are 
exponentially growing Fourier weights in the sense of 
\citet[Sec.~2.1]{GneEtAl24}. 
Hence we obtain from \citet[Cor.~4.10]{GneEtAl24} that the polynomial 
decay rate of $n$-th minimal errors of integration on the 
Hermite space $H(K_\bb)$ is 
\[
\dec(K_\bb) = \frac{1}{2} (\widetilde{\rho} - 1),
\]
where
\[
\widetilde{\rho} := \liminf_{j\to \infty} \frac{r_j \cdot \ln(2)}{\ln(j)}.
\]
Given \eqref{g20} we have $\beta_j \asymp \sigma_j^2$, and therefore
\[
\rho = \liminf_{j\to \infty} \frac{\ln(1/\beta_j)}{\ln(j)} = 
\widetilde{\rho}.
\qedhere
\] 
\end{proof}

Algorithms that yield $\dec(L_\bs) \geq (\rho - 1)/2$
are obtained by applying the transference result to properly chosen
multivariate decomposition methods (MDMs) on the space $H(K_\bb)$.
Here we sketch the construction of the MDMs,
see \citet[Sec.~4.3]{GneEtAl24} for details and further references.

Let $\bU$ denote the set of all finite subsets of $\N$. 
For $\bx \in \R^\N$ and $\bv \in \bU$ we use $(\bx_\bv, \bn_{\bv^c})$ 
to denote the point $\by \in \R^\N$ given by $y_j := x_j$ if $j\in \bv$ 
and $y_j=0$ otherwise. Due to Lemma~\ref{l6},
we always have $(\bx_\bv, \bn_{\bv^c}) \in \X(K_\bb)$.
The construction of MDMs 
is based on the anchored decomposition
\begin{equation}\label{g81}
f = \sum_{\bu \in \bU} f_\bu
\end{equation}
of the functions $f \in H(K_\bb)$, where
\[
f_\bu(\bx) = 
\sum_{\bv\subseteq \bu} (-1)^{|\bu\setminus \bv|} \cdot
f(\bx_\bv, \bn_{\bv^c})
\]
for every $\bx \in \X(K_\bb)$. The anchored components
$f_\bu$ of $f$ only depend on the variables with indices
in $\bu$, so that $f_\bu$ may be considered as a multivariate
function with domain $\R^{|\bu|}$. We stress that \eqref{g81} is
a non-orthogonal decomposition in the space $H(K_\bb)$.

At first, we choose a finite set $\cB$ of non-empty elements of
$\bU$, which specifies the anchored components $f_\bu$,
whose integrals with respect to the $|\bu|$-dimensional standard
normal distribution will be approximated.
For each $\bu \in \cB$ we employ a Smolyak formula $B_{\bu,n_\bu}$,
which uses $n_\bu \in \N$ function values for the integration 
of the functions $f_\bu$.
 The corresponding MDM is given by
\begin{equation}\label{eq:form_MDM}
B (f) := 
f(\bn) + 
\sum_{\bu \in \cB} B_{\bu,n_\bu}(f_\bu).
\end{equation}
As building blocks for the Smolyak formulas, we use
a sequence $(B_n)_{n\in\N}$ of univariate quadrature formulas,
where $B_n$ uses $n$ function evaluations and $e(B_n,k_{\beta_1})$
converges to 0 sufficiently fast, see again \citet[Sec.~4.3]{GneEtAl24} for details.
For instance, we may use Gauss-Hermite rules $B_n$, which yield
\[
e(B_n,k_{\beta_1})\leq C\cdot\beta_1^n
\]
with an explicitly known absolute constant $C$, see \citet[Prop.~1]{IKLP15} in the
specific case $\omega:=1/e$ as well as $a:=\ln(1/\beta_1)$ and $b:=1$.
For alternatives, see Section~\ref{s4.2}.
For the specific choice of $\cB$ and of $n_\bu$ for $\bu \in \cB$
we refer again to \citet[Sec.~4.3]{GneEtAl24}.

\section{$L^2$-Approximation}\label{s3}

As before, we consider sequences $\bs := (\sigma_j)_{j\in J}$ and
$\bb := (\beta)_{j\in J}$ of shape parameters and base parameters, 
respectively.

\subsection{The Transference Result}

In contrast
to integration, we may use the relation \eqref{g2} directly to establish
a transference result for $L^2$-approximation in a straightforward
way.

\begin{theo}\label{t2}
Assume that \eqref{g1a} is satisfied. Moreover, assume that
\eqref{g2} and \eqref{g3} are satisfied for every $j \in J$.
For every linear sampling method $A$ with nodes from $\ell^2(\bs^2)$ we
have
\[
e(A,L_{\bs}) = (\be-\bb)_\ast^{1/2} \cdot e(Q_\bc^{-1} A Q_\bc , K_\bb).
\]
\end{theo}

\begin{proof}
We apply Theorem~\ref{t1}. Put $Q := Q_\bc$.
Since $Q_{|H(K_\bb)}$ is an isomorphism from
$H(K_\bb)$ to $H(L_{\bs})$ with
\[
\Vert Q f\|^2_{H(L_\bs)} =
\Vert f\|^2_{H((\be-\bb)_\ast K_\bb)} = (\be-\bb)_\ast^{-1} \cdot
\Vert f\|^2_{H(K_\bb)}
\]
for every $f \in H(K_\bb)$, we have
\[
e(A,L_{\bs}) = 
(\be-\bb)_\ast^{1/2} \cdot 
\sup_{\Vert f \Vert_{H(K_\bb)} \leq 1} 
\Vert Q f - A Qf \Vert_{L^2(\mu)}.
\]
Furthermore,
\[
\Vert Q f - A Qf \Vert_{L^2(\mu)} = \Vert f - 
Q^{-1} A Q f \Vert_{L^2(\mu)},
\]
since $Q$ is an isometric isomorphism on $L^2(\mu)$.
\end{proof}

\begin{rem}\label{r5}
Let the assumptions from Theorem~\ref{t2} be satisfied.
For a linear sampling method $A$ on the Gaussian space $H(L_\bs)$
according to \eqref{g5} with nodes from $\ell^2(\bs)$
the linear sampling method $Q_\bc^{-1} A Q_\bc$
on the Hermite space $H(K_\bb)$ is easily
determined explicitly: For $f \colon \ell^2(\bs^2) \to \R$ we have
\[
Q_\bc^{-1} A Q_\bc (f) =
\sum_{i=1}^n Q_\bc f (\bx_i) \cdot Q_\bc^{-1} a_i =
\sum_{i=1}^n f (\by_i) \cdot b_i 
\]
with $\by_i \in \ell^2(\bs^2)$ and $b_i \in L^2(\mu)$ given by
\[
\by_i := \bc \bx_i 
\qquad \text{and} \qquad
b_i (\bz) := 
\frac{\phi_\bc(\bx_i)}{\phi_\bc(\bc^{-1} \bz)} \cdot a_i(\bc^{-1} \bz),
\]
see Lemma~\ref{l1}. 

It is easy to see
that $(\bx_i,a_i)\mapsto(\by_i,b_i)$ defines a bijection on the set
$\ell^2(\bs^2)\times L^2(\mu)$ of pairs of nodes and coefficients.
Furthermore, 
\[
\cost(A) = \cost(Q_\bc^{-1} A Q_\bc),
\]
cf.\ Remark~\ref{r7}.
\end{rem}

\begin{cor}\label{c1}
Assume that \eqref{g1a} and \eqref{g2} are satisfied.
For every $n \in \N$ we have
\[
\frac{e_n(L_{\bs})}{e_0(L_\bs)} = \frac{e_n(K_{\bb})}{e_0(K_\bb)} =
e_n(K_\bb),
\]
i.e., the normalized $n$-th minimal errors for $L^2$-approximation on 
Gaussian spaces and on the corresponding Hermite spaces, related by 
\eqref{g2}, coincide.
\end{cor}

\begin{proof}
Combine Theorem~\ref{t2}, Remark~\ref{r5}, and Lemma~\ref{l7}.
\end{proof}

\subsection{Functions of Finitely Many Variables}

$L^2$-approximation in the case of finitely many variables has been 
studied previously by \citet{FHW2012} and \citet{SW2018} for Gaussian 
spaces and by \citet{IKPW16a} and \citet{IKPW16b} for Hermite spaces. 
While the latter paper employs linear sampling methods as described
in Section~\ref{s2a}, \citet{IKPW16a} and \citet{SW2018} allow a larger 
class of
algorithms, namely linear algorithms $A$ based on linear information,
i.e.,
\begin{equation}\label{eq:linear_algorithm}
A(f) := \sum_{i=1}^n \lambda_i(f) \cdot a_i
\end{equation}
with $n\in\N$, bounded linear functionals $\lambda_i$ on $H(M)$
and coefficients $a_i\in L^2(\mu)$; \citet{FHW2012} consider both 
classes of algorithms. In the case of general linear algorithms,
the same tractability concepts have been 
studied on the one hand in \citet{FHW2012} and \citet{SW2018} for
Gaussian spaces, and on the other hand in 
\citet{IKPW16a} for Hermite spaces, with rather complete and
matching results for both types of function space.

It was already observed in \citet{IKPW16b} that the linear sampling
methods studied there achieve the same results as
algorithms of the form \eqref{eq:linear_algorithm} for several 
tractability concepts.
Indeed, as shown in \citet{kriegetat23}, the results from \citet{DKM22} are applicable and lead to rather complete results
for tractability 
based on linear sampling methods, though in a non-constructive way.
Therefore, we do not apply our transference result to this case here.
Our results for infinitely many variables, however,
are new.  

\subsection{Functions of Infinitely Many Variables}
\label{SUBSEC:FwIMV}

As in Section~\ref{SUBSEC:integr_IMV}, we consider $\rho$ given by
\eqref{eq:rho} and assume that the function $\$$ obeys \eqref{g82}.

The next result follows easily from \citet[Cor.~4.10]{GneEtAl24} 
by means of our transference result.

\begin{theo}\label{Thm:InfDimApp}
Assume that \eqref{g1a} is satisfied.
The polynomial decay rate of $n$-th minimal errors of 
$L^2$-ap\-prox\-i\-ma\-tion on the Gaussian space $H(L_\bs)$ is 
\[
\dec(L_\sigma) = \frac{1}{2} (\rho - 1).
\]
\end{theo}

\begin{proof}
We proceed as in the proof of Theorem~\ref{Thm:InfDimInt}: 
By Corollary~\ref{c1}, we only have to
consider $\dec(K_\bb)$, and $\dec(K_\bb)$ is known for
$L^2$-approximation, too.
This time, however,
the sequence $\bb = (\beta_j)_{j\in\N}$ is defined as in \eqref{g2},
but again $\beta_j \asymp \sigma_j^2$
and the weights $\beta_j^{-\nu} = 2^{r_j \cdot \nu}$ with
$\nu,j\in\N$ are exponentially growing Fourier weights in the sense of 
\citet[Sec.~2.1]{GneEtAl24}, possibly after reordering the kernels.
As before,
\[
\rho = \liminf_{j\to \infty} \frac{\ln(1/\beta_j)}{\ln(j)} = 
\liminf_{j\to \infty} \frac{r_j \cdot \ln(2)}{\ln(j)}.
\] 
Hence we obtain from \citet[Cor.~4.10]{GneEtAl24} that the polynomial 
decay rate of $n$-th minimal errors of $L^2$-approximation on the 
Hermite space $H(K_\bb)$ is 
\[
\dec(K_\bb) = \frac{1}{2} (\rho - 1).
\qedhere
\]
\end{proof}

{\small

\bibliographystyle{abbrvnat}
\bibliography{references}

}

\end{document}